\documentclass[12pt,oneside,reqno]{amsart}
\usepackage{amsmath,amsthm,amssymb,epic,eepic,bm}
\usepackage{eqlist,eqparbox}
\usepackage{comment}

\title {\bf On the extensions of Di Nola's Theorem}
\author {Michal Botur}
\author {Jan Paseka}
\thanks{Both authors gratefully acknowledge  the support by ESF Project CZ.1.07/2.3.00/20.0051
Algebraic methods in Quantum Logic of the Masaryk University. 
M. Botur gratefully acknowledges Financial Support 
of the  the Grant Agency of the Czech
Re\-pub\-lic un\-der the grant
No.~   GA\v CR P201/11/P346.}
\address{Palack\' y University Olomouc, Faculty of Sciences, t\v r. 17.listopadu 1192/12, Olomouc 771 46, Czech Republic}
\email{michal.botur@upol.cz}
\address{Department of Mathematics and Statistics, Faculty of Science,
Masaryk University, {Kotl\'a\v r{}sk\' a\ 2}, 611~37 Brno, 
Czech Republic}
\email{paseka@math.muni.cz}

\begin{document}

\newtheorem{defin}{Definition}
\newtheorem{prop}{Proposition}
\newtheorem{theorem}{Theorem}
\newtheorem{Lemma}{Lemma}
\newtheorem{cor}{Corollary}
\newtheorem{remark}{Remark}
\newtheorem{cl}{Claim}

\newcommand{\zl}{\mbox{$[\hspace{-1.8pt} [$}}
\newcommand{\zr}{\mbox{$]\hspace{-2pt} ]$}}

\begin{abstract}
The main aim of this paper is to present a direct proof of Di Nola's representation Theorem for MV-algebras and to extend 
his results to the restriction of the standard MV-algebra on 
rational numbers. The results are based on a direct proof of the theorem which says that any finite partial subalgebra of a linearly ordered  MV-algebra can be embedded into $\mathbb Q\cap [0,1].$
\end{abstract}

\maketitle

\section*{Introduction}

MV-algebras was introduced by Chang \cite{Chan} as algebraic opposite of \L ukasiewicz multivalued propositional logic. The main idea of his definition is to present logic with truth scale $[0,1]\subseteq\mathbb{R}$ where basic connective disjunction is represented by the cut addition
$$x\oplus y:= \min\{x+y,1\}$$    
and negation is defined as antitone involution by $\neg x:=1-x.$ This model is called as the standard MV-algebra. The \L ukasiewicz multivalued logic (and thus also MV-algebras) became very popular in applications of fuzzy logics for its simplicity and its nature. 

The theory of MV-algebras is highly developed and we have many interesting results and connections with another important parts of mathematics. First, the category of MV-algebras is equivalent with category of commutative $\ell$- groups (lattice ordered commutative groups). 
Second, the variety (or quasivariety) of MV-algebras is generated by its standard model. Consequently, the free algebras are subalgebras of direct sum $[0,1]^X$ of standard MV-algebra and thus the MV-algebras are related with some basic geometrical theories (see \cite{Mun2}). But most of the very deep results are dependent on some of representation theorems. 

The representation theory of MV-algebras is based on 
Chang's representation Theorem \cite{Chan}, McNaughton's Theorem and 
Di Nola's representation Theorem \cite{DiNola}.  Chang's representation Theorem 
yields a subdirect representation of all MV-algebras via linearly ordered MV-algebras. 
McNaughton's Theorem characterizes free MV-algebras as algebras 
of continuous, piece-wise linear functions
with integer coefficients on $[0,1]$. Finally, Di Nola's 
representation Theorem describes MV-algebras as sub-algebras of algebras 
of functions with values into a non-standard ultrapower of
the MV-algebra $[0, 1]$.

The main motivation for our paper comes from the fact that although 
the proofs of both Chang's representation Theorem [4] and  
McNaughton's Theorem are of algebraic nature 
the proof of Di Nola's representation Theorem is based on model-theoretical considerations. 
We give a simple, purely algebraic, proof of it and its variants 
based on the 
Farkas' Lemma for rationals \cite{farkas} and  General finite embedding theorem \cite{Bot}. 
In addition, we present a uniform  version of our results.

\section{Preliminaries}
\subsection{MV-algebras}
Recall, that by an {\it MV-algebra} is meant an algebra $\mathbf A=(A,\oplus,\neg,0)$ of type $(2,1,0)$ satisfying the axioms:
\begin{itemize}
\item[(MV1)] $x\oplus y = y\oplus x,$
\item[(MV2)] $x\oplus (y\oplus z) =(x\oplus y)\oplus z,$
\item[(MV3)] $x\oplus 0=x,$
\item[(MV4)] $\neg\neg x=x,$
\item[(MV5)] $x\oplus 1=1,$ where $1:=\neg 0,$
\item[(MV6)] $\neg(\neg x\oplus y)\oplus y = \neg (\neg y\oplus x)\oplus x.$
\end{itemize} 

The order relation $\leq$ can be introduced on any MV-algebra $\mathbf A$ by 
the stipulation $$x\leq y\mbox{ if and only if } \neg x\oplus y =1.$$ Moreover, the 
ordered set $(A,\leq)$ can be organized into a bounded lattice 
$(A,\vee,\wedge,0,1)$ where $$x\vee y = \neg (\neg x\oplus y)\oplus y \mbox{ and } x\wedge y = \neg(\neg x\vee \neg y).$$

In theory of MV-algebras we use derived operations $\odot$ and $\rightarrow$ defined by $x\odot y:= \neg (\neg x\oplus \neg y)$ and $x\rightarrow y:= \neg x \oplus y.$ Those operations are connected by adjointness property $$x\odot y\leq z\mbox{ if and only if } x\leq y\rightarrow z.$$ 

We recall that a {\em filter} of an MV-algebra  is a non-empty set which contains top element 1, it is closed upwards and under the operation $\odot.$ The filters are just kernels of congruences, the variety of MV-algebras is 1-regular\footnote{There is one-to-one correspondence between kernels and congruences. The lattice of congruences is isomorphic to the lattice of filters.}. A {\em prime filter} is just such a filter $P$ which satisfies $x\rightarrow y\in P$ or $y\rightarrow x\in P$ for any $x$ and $y.$ The kernel of congruence is prime filter if and only if the factor MV-algebra is linearly ordered. We remark that the intersection of all prime filters is just the filter $\{ 1\}$ and thus the MV-algebras are subdirect products of MV-chains (see \cite{Mun}). 
An {\em ultrafilter}  of an MV-algebra is a maximal proper filter.

\subsection{Generalized finite embedding theorem}

By an {\em ultrafilter} on a set $I$ we mean an ultrafilter of the Boolean algebra  $\mathcal P (I)$ of the subsets of $I.$

    Let $\{\mathbf{A}_i;\ i\in I\}$ be a system of algebras of the same type $\mathrm{F}$ for $i\in I$. We denote for any $x,y\in \prod_{i\in I}A_i$ the set

    $$\zl x=y\zr=\{j\in I;x(j)=y(j)\}.$$
    
    \noindent If $F$ is a filter of $\mathcal{P}(I)$ then the relation $\theta_F$ defined by

    $$\theta_F=\{\langle x,y\rangle\in (\prod_{i\in I}A_i)^2; \zl x=y\zr \in F \}$$

    \noindent is a congruence on $\prod_{i\in I}\mathbf{A}_i$. 
    For an ultrafilter $U$ of
    $\mathcal{P}(I)$, an algebra $ (\prod_{i\in I}\mathbf{A}_i)/U := (\prod_{i\in
    I}\mathbf{A}_i)/\theta_U$ is said to be an \textit{ultraproduct} of algebras $\{\mathbf{A}_i;\ i\in I\}$. Any ultraproduct of an algebra $\mathbf A$ is called an ultrapower of $\mathbf A$. The class of all ultraproducts 
(products, isomorphic images) of algebras from some class of 
algebras $\mathcal K$ is denoted by $\mathrm{P_U}(\mathcal K)$ ($\mathrm{P}(\mathcal K)$, $\mathrm{I}(\mathcal K)$). The class of all finite algebras from some class of algebras $\mathcal K$ is denoted by $\mathcal K_{Fin}.$

 \begin{defin}{\rm  Let $\mathbf{A}=(A,\mathrm{F})$ be 
a partial algebra and $X\subseteq A$. Denote
the partial algebra $\mathbf{A}|_X=(X,\mathrm{F})$, where for any $f\in F_n$ and all $x_1,\ldots,x_n\in
        X$, $f^{\mathbf{A}|_X}(x_1,\ldots,x_n)$ is defined if and only if
         $f^{\mathbf{A}}(x_1,\ldots,x_n)\in X$
        holds. Moreover, then we put
        $$
        f^{\mathbf{A}|_X}(x_1,\ldots,x_n) := f^{\mathbf{A}}(x_1,\ldots,x_n).
        $$}
    \end{defin}

\begin{defin}{\rm 
An algebra $\mathbf{A}=(A,\mathrm{F})$ satisfies 
{\em the general finite embedding} ({\em finite embedding property})
{\em property for the class $\mathcal{K}$} of algebras of 
the same type if for any finite subset $X\subseteq A$ 
there are an (finite) algebra $\mathbf{B}\in\mathcal{K}_\mathrm{E}$ 
and an embedding $\rho:\mathbf{A}|_X \hookrightarrow \mathbf{B}$, i.e. an injective mapping $\rho:X\rightarrow B$ satisfying the property $\rho (f^{\mathbf{A}|_X}(x_1,\ldots,x_n))= f^\mathbf{B}(\rho(x_1),\ldots,\rho (x_n))$ if  $x_1,\dots,x_n\in X$, $f\in\mathrm{F}_n$ and $f^{\mathbf{A}|_X}(x_1,\ldots,x_n)$ is defined.}
    \end{defin}

Finite embedding property is usually denoted by (FEP). 
Note also that  a quasivariety ${\mathcal K}$ has the FEP if and only 
if ${\mathcal K}= \mathrm{ISPP_U}(\mathcal{K}_{Fin})$ 
(see \cite[Theorem 1.1]{Blok2} or \cite{Blok}).

    \begin{theorem}\cite[Theorem 6]{Bot}\label{gep}
        Let $\mathbf{A}=(A,\mathrm{F})$ be a  algebra and let
        $\mathcal{K}$ be a class of algebras of the same type. If
        $\mathbf{A}$ satisfies the general finite embedding property for $\mathcal{K}$ then
        $\mathbf{A}\in \mathrm{ISP_U}(\mathcal{K})$.
    \end{theorem}

\begin{theorem}\cite[Theorem 7]{Bot}
        Let $\mathbf{A}=(A,\mathrm{F})$ be an algebra such that $\mathrm{F}$ is finite and let
        $\mathcal{K}$ be a class of algebras of the same type. If
        $\mathbf{A}\in \mathrm{ISP_U}(\mathcal{K})$
         then  $\mathbf{A}$ satisfies the general finite embedding property for $\mathcal{K}$.
    \end{theorem}

\subsection{Farkas' lemma}

Let us  recall the original formulation of Farkas' lemma \cite{farkas, schrijver} on rationals:

\begin{theorem}[Farkas' lemma] \label{Farkas} 
Given a matrix $A$  in ${\mathbb Q}^{m \times n}$ and 
$\bm{c}$ a column vector in ${\mathbb Q}^{m}$, then
there exists  a column vector $\bm{x}\in {\mathbb Q}^{n}$, 
$\bm{x}\geq \bm{0}_n$ and $A\cdot \bm{x}= \bm{c}$ if and only if, 
for all  row vectors $\bm{y}\in {\mathbb Q}^{m}$,  
 $\bm{y}\cdot A \geq  \bm{0}_m$ implies 
$\bm{y}\cdot \bm{c} \geq 0$.
\end{theorem}
In what follows, we will use the following equivalent formulation:

\begin{theorem}[Theorem of alternatives] \label{alt} Let $A$ be a matrix in ${\mathbb Q}^{m \times n}$ and 
$\bm{b}$ a column vector in ${\mathbb Q}^{n}$. The system $A\cdot \bm{x} \leq \bm{b}$ has no solution if and only if there exists
a row vector $\bm{\lambda}\in {\mathbb Q}^{m}$ such that $\bm{\lambda}\geq \bm{0}_m$, 
$\bm{\lambda}\cdot A = \bm{0}_n$ and $\bm{\lambda}\cdot \bm{b} < 0$.
\end{theorem}

\begin{remark}{\rm Since the row vector $\bm{\lambda}\in {\mathbb Q}^{m}$ from Theorem \ref{alt} has non-negative 
rational components $\lambda_i=\frac{p_i}{q_i}$, $p_i\in \mathbb N_0$, $q_i\in \mathbb N$ 
we may assume (by taking  the least 
common multiple $q$ of denominators $q_i$ and multiplying by it the respective conditions for  $\bm{\lambda}$) that 
$\bm{\lambda} \in {\mathbb Z}^{m}$.}
\end{remark}

\section{The Embedding Lemma}

In this section, we use the Farkas' lemma on rationals to prove 
that any finite partial subalgebra of a linearly ordered  MV-algebra can be 
embedded into $\mathbb Q\cap [0,1]$ and hence into the finite MV-chain 
$\mathbf L_k\subseteq [0,1]$ for a suitable $k\in \mathbb N$.

\begin{Lemma}\label{1}
Let $\mathbf M=(M;\oplus,\neg,0)$ be a linearly ordered MV-algebra, $X\subseteq
M\setminus \{0\}$ 
be a finite subset. Then there is a rationally valued map
$s:X\cup\{0,1\}\longrightarrow [0,1]\cap\mathbb{Q}$ such that
\begin{enumerate}
\item $s(0)=0,$ $s(1)=1,$
\item if $x,y, x\oplus y\in X\cup\{0,1\}$ such that $x\leq \neg y$ and $x,y\in
X\cup\{0,1\}$ then $s(x\oplus y)=s(x)+s(y).$
\item if $x\in X$ then $s(x)>0.$
\end{enumerate}
\end{Lemma}
\begin{proof}
We put $$Y(X):=\{x\oplus y\mid x,y\in X\cup\{0,1\}\}.$$ 

Thus $Y(X)\subseteq M$ is finite, $X\subseteq Y(X),$ $0,1\in Y(X),$ $X\oplus
X\subseteq Y(X)$. Since $\mathbf M$ is a chain we may assume that
$Y(X)=\{y_0=0<y_1<\cdots<y_n=1\}$ and put 
$\bm{y}=(y_1, \dots, y_n)^{T}\in Y(X)^{n}$. 
For any $x\in X$ there is an index $1\leq j_x\leq
n$ such that $x=y_{j_x}.$ If $x\leq \neg y, x\not=y$ and $x,y\in X$ we denote by
$\mathbf a_{x,y}\in \mathbb{Z}^n$ a row vector such that 
$$ \mathbf a_{x,y}(j) =\left\{
            \begin{array}{rl}
                1 & \mbox{if } j=j_x\ \mbox{or}\ j=j_y\\
                -1 & \mbox{if } j=j_{x\oplus y}\\
                0 & \mbox{otherwise.}
            \end{array}
        \right.$$

If $x\leq \neg x$ and $x\in X$ we denote by 
$\mathbf a_{x,x}\in \mathbb{Z}^n$ a row vector such
that 
$$ \mathbf a_{x,x}(j) =\left\{
            \begin{array}{rl}
                2 & \mbox{if } j=j_x\\
                -1 & \mbox{if } j=j_{x\oplus x}\\
                0 & \mbox{otherwise.}
            \end{array}
        \right.$$
Let $A$ be a matrix consisting of rows $\mathbf a_{x,y}$ 
such that $x\leq \neg y$ and $x,y\in X.$ Let $E_n$ be the 
identity matrix of order $n$. We put $m=|\{\mathbf a_{x,y} \mid x,y\in X, x\leq \neg
y\}|$. Let 
us denote by $(*)$ the following system of linear inequalities with variables
$z_1,\cdots, z_n$ over rationals:
$$\left ( \begin{array}{c}
-E_n\\ \phantom{-}A\\ -A
\end{array} \right )\cdot 
\left( \begin{array}{c} z_1\\ \vdots \\ z_n
\end{array}\right)\leq 
\left( \begin{array}{c}
\mathbf{-1}_n\\ \mathbf 0_m\\ \mathbf{0}_m 
\end{array}\right). \eqno{(*)}
$$

Then by the Farkas' lemma (see Theorem \ref{alt}) for rationals the systems of inequalities $(*)$ does not
have a solution in $\mathbb Q^n$ if and only if there is a row vector
$\bm{\lambda}=(\lambda_1, \dots, \lambda_{n+2m}) \in\mathbb Z^{n+2m}$, 
$\bm{\lambda}\geq \mathbf 0_{n+ 2m}$ such that 
$$
\bm{\lambda}\cdot\left( \begin{array}{c}
-E_n\\ \phantom{-}A\\ -A
\end{array} \right)=0,\, \bm{\lambda} \cdot \left( \begin{array}{c}
\mathbf{-1}_n\\ \mathbf 0_m\\ \mathbf{0}_m 
\end{array}\right) <0. \eqno{(**)}
$$ Assume that the vector $\bm{\lambda}\in\mathbb Z^{n+2m}$ satisfying $(**)$ exists. 
Hence there is an index $1\leq j_0\leq n$ such $\lambda_{j_0}>0.$

Due to Chang's Theorem our linearly ordered MV-algebra $\mathbf M$ is an 
interval $[0,u]$ in a linearly ordered commutative $\ell$-group $\mathbf G$ with 
a strong unit $u.$ If $u, v\in M$ are such that $u\leq \neg v$ 
then the sum $u\oplus v$ coincides with the sum $u+v$ computed in $\mathbf G.$ 
We then have 

$$0 = \bm{\lambda}\cdot\left( \begin{array}{c}
-E_n\\ \phantom{-A}\\ -A
\end{array} \right)\cdot
\left( \begin{array}{c} y_1\\ \vdots \\ y_n
\end{array}\right) = \bm{\lambda}\cdot \left( \begin{array}{c} -\mathbf y \\ \mathbf
0_n \\ \mathbf 0_n
\end{array}\right)=-\sum_{j=1}^n \lambda_j y_j. \eqno{(***)}
$$
Because $\lambda_1, \cdots, \lambda_{n+2m}$ are non-negative and $\lambda_{j_0}$ is
positive,  
moreover $y_1,\cdots, y_n$ are also positive non-zero elements in $\mathbf G,$ we 
get that $\sum_{i=1}^n \lambda_jy_j$ is  a positive non-zero element from $\mathbf G$
which is a contradiction with $(***)$. 

It follows that the system $(*)$ has a rational valued solution $(q_1,\cdots ,q_n)$ 
and from $(*)$ it clearly follows that the solution is  positive (more precisely 
$\mathbf q=(q_1,\cdots,q_n) \geq \mathbf 1_n$). We define the map 
$s:X\cup \{0,1\}\longrightarrow [0,1]\cap\mathbb Q$ by the following prescription:
$$s(x)=\left\{\begin{array}{ll}
{q_{j_x}\over q_n} & \mbox{if } x\in X,\\
0 &\mbox{if } x=0,\\
1 &\mbox{if } x=1.
\end{array}\right.$$
The mapping $s$ evidently satisfies the conditions (1)-(3) of this Lemma.
\end{proof}

\begin{Lemma}[{\bfseries Embedding Lemma}]\label{2}
Let us have a linearly ordered MV-algebra $\mathbf M=(M;\oplus,\neg,0)$ and let
$X\subseteq M$ be a finite set. Then there exists an embedding 
$f:\mathbf X\hookrightarrow \mathbf L_k,$ where $\mathbf X$ is 
a partial MV-algebra obtained by the restriction of\/ $\mathbf M$ to 
the set $X$ and\/ $\mathbf L_k\subseteq [0,1]$ is the
linearly ordered finite MV-algebra on the set $\{0,{1\over k}, {2\over k},\cdots ,
1\}.$ 
\end{Lemma}
\begin{proof}
Let us define a set $Y$ as follows: 
$$Y:=\{x\ominus y\mid x,y\in X\cup\{0,1\}\}\setminus \{0\}.$$
Moreover, let $s:Y\cup \{0, 1\}\longrightarrow [0,1]\cap\mathbb Q$ be the respective mapping for the set $Y$ 
from Lemma \ref{1}.

Let $f=s|_X$ be the restriction of the mapping $s$ on the set $X.$ 

Let $f(X)\setminus \{0\}=\{{{p_1}\over{q_1}}, \dots,{ {p_l}\over {q_l}} \}$ for some 
$p_1, q_1, \dots, p_l, q_l\in \mathbb N$ and let us denote by 
$k$ the least common multiple of the denominators $q_1, \dots, q_l$. Then evidently 
$f(X)\subseteq \{0,{1\over k}, {2\over k},\cdots , 1\}$.

If $0\in X$ then 
by definition of $s$ we obtain $f(0)=s(0)=0.$ If $x,\neg x\in X$ for 
some $x\in M$ then clearly $\neg x\leq \neg x$ and using Lemma \ref{1} 
we obtain $s(\neg x)+ s(x)=s(\neg x \oplus x)=s(1)=1.$ 
Hence, $f(\neg x)=s(\neg x)=1-s(x)=\neg s(x)=\neg f(x).$ 

Finally, let $x,y\in X$ be such that $x\oplus y\in X$. Then
\begin{enumerate}
\item if $x\leq \neg y$ then $f(x\oplus y)= s(x\oplus y)=s(x) + s(y)= f(x)\oplus f(y).$
\item if $\neg y < x$ then $f(x\oplus y)= f(1)=s(1)=1.$ Conversely, $x\ominus \neg
y\in Y$ and $x\ominus \neg y = y\ominus \neg x\leq y$ and thus $s(x)= s((x\ominus
\neg y)\oplus \neg y) = s(x\ominus \neg y)+ s(\neg y)\geq s(\neg y)=1-s(y).$ 
It follows that $1\geq f(x)\oplus f(y)=\mbox{min}(1, s(x)+s(y))\geq \mbox{min}(1,
(1-s(y))+s(y))=1$.
\end{enumerate}
\end{proof}

\section{Extensions of Di Nola's Theorem}

In this section, we are going to show 
Di Nola's representation Theorem and its several variants not only 
via standard MV-algebra $[0,1]$ but also via its rational part 
$\mathbb Q\cap[0,1]$ and 
finite MV-chains. To prove it, we use the Embedding Lemma obtained in the 
previous section. First, we establish the FEP for 
linearly ordered MV-algebras.

\begin{theorem}\label{T-1}
\begin{enumerate}
\item The class $\mathcal{LMV}$ of linearly ordered MV-algebras has the FEP.
\item  The class $\mathcal{MV}$ of MV-algebras has the FEP.
\end{enumerate}
\end{theorem}
\begin{proof}
1) It follows immediately from Lemma \ref{2}.

2) Let $\mathbf M=(M;\oplus,\neg,0)$ be an MV-algebra and let $X\subseteq M$ be a finite subset. For any $x,y\in X; x\not = y$ there is a prime filter $P$ such that $x/P\not= y/P$. Hence there is a finite system of prime filters $P_1,\ldots, P_l$ such that it separates elements from $X,$ i.e., $\mathbf X\hookrightarrow \prod_{i=1}^l (\mathbf M/P_i)$ is an injective mapping. For any $i\in \{1,\ldots, l\}$ there is 
by Lemma  \ref{2} an embedding 
$f_i:\mathbf X/P_i\hookrightarrow \mathbf L_{k_i}.$   Let $k$ be the least common multiple of $k_1,\ldots , k_l.$ Thus, for any $i\in \{1,\ldots, l\}$ there is an embedding $f_i:\mathbf X/P_i\hookrightarrow \mathbf L_{k}.$ Consequently there is an embedding $f:\mathbf X\hookrightarrow (\mathbf L_k)^l$ defined by $f(x)(i)=f_i(x/P_i).$
\end{proof}

Note that the part (1) of the preceding theorem for 
subdirectly irreducible MV-algebras can be 
easily deduced from the result 
that the class of subdirectly irreducible Wajsberg hoops has the FEP 
(see \cite[Theorem 3.9]{Blok}). The well-known part (2) then follows 
from \cite[Lemma 3.7,Theorem 3.9]{Blok}. We are now ready to establish 
a variant of Di Nola's representation Theorem for finite MV-chains 
(finite MV-algebras).

\begin{theorem}\label{T2}
\begin{enumerate}
\item Any linearly ordered MV-algebra can be embedded into an ultraproduct of finite MV-chains.
\item Any MV-algebra can be embedded into a product of ultraproducts of finite MV-chains.
\item Any MV-algebra can be embedded into an ultraproduct of finite MV-algebras (which are embeddable into powers of finite MV-chains).
\end{enumerate}
\end{theorem}
\begin{proof}
(1) It is a direct corollary of Theorem \ref{T-1}, (1) and Theorem \ref{gep}.

(2) Any MV-algebra is embeddable into a product of linearly ordered ones. The rest follows by (1).

(3) It is a direct corollary of Theorem \ref{T-1}, (2) and Theorem (1).

\end{proof}

The next two theorems cover Di Nola's representation Theorem and 
its respective variants both for 
rationals and reals.

\begin{theorem}\label{T1}
\begin{enumerate}
\item Any linearly ordered MV-algebra can be embedded into an ultrapower of $\mathbb Q\cap[0,1].$
\item Any MV-algebra can be embedded into a product of ultrapowers of $\mathbb Q\cap[0,1].$
\item Any MV-algebra can be embedded into an ultrapower of the countable power of $\mathbb Q\cap[0,1].$
\item  Any MV-algebra can be embedded into an ultraproduct of finite powers of $\mathbb Q\cap[0,1].$
\end{enumerate}
\end{theorem}

\begin{proof}
(1)-(4) It is a corollary of Theorem \ref{T2}.
\end{proof}

\begin{theorem}\label{T0}
\begin{enumerate}
\item Any linearly ordered MV-algebra can be embedded into an ultrapower of $[0,1].$
\item Any MV-algebra can be embedded into a product of ultrapowers of $[0,1].$
\item Any MV-algebra can be embedded into an ultrapower of the countable  power of $[0,1].$
\item  Any MV-algebra can be embedded into an ultraproduct of finite powers of $[0,1].$
\end{enumerate}
\end{theorem}
\begin{proof}
(1)-(4) It is a corollary of Theorem \ref{T2}.
\end{proof}

\section{General finite $\alpha$-embedding property}

This paragraph is devoted to a general finite $\alpha$-embedding 
theorem which is necessary for proving the uniform variants 
of Di Nola's theorem. At first we recall some definitions.


\begin{defin}\cite{Keisler} {\rm Let $\alpha$ be an infinite  
cardinal.  A proper filter 
$D$ over $I$  is said to be
{\em $\alpha$-regular} if there exists a set 
$E \subseteq D$ such that $|E| = \alpha$ and 
each $i \in I$ belongs to only finitely many $e \in E$.}
\end{defin}

\begin{defin} {\rm Let $\alpha$ be an infinite  cardinal, 
$\mathbf{A}=(A,\mathrm{F})$ an algebra such that $|A| \leq \alpha$.  
Let $i_A:A\to \alpha$ be the respective injective mapping.
$\mathbf{A}$ satisfies the {\em general finite 
$\alpha$-embedding} ({\em finite $\alpha$-embedding property})
{\em property for the class $\mathcal{K}$} of algebras of 
the same type if for any finite subset $X\subseteq \alpha$ there 
are an (finite) algebra $\mathbf{B}\in\mathcal{K}$ and an embedding $\rho:\mathbf{A}|_{i_{A}^{-1}(X)}\hookrightarrow\mathbf{B}$, i.e. 
an injective mapping $\rho:{i_{A}^{-1}(X)}\rightarrow B$ 
satisfying the property 
$\rho (f^{{\mathbf{A}|_{i_{A}^{-1}(X)}}}(x_1,\ldots,x_n))= 
f^\mathbf{B}(\rho(x_1),\ldots,\rho (x_n))$ if  
$i_{A}(x_1),\dots,i_{A}(x_n)\in X$, $f\in\mathrm{F}_n$ and $f^{\mathbf{A}|_{i_{A}^{-1}(X)}}(x_1,\ldots,x_n)$ is defined.}
\end{defin}

The following theorem is an extension of \cite[Theorem 6]{Bot} for 
algebras of a bounded cardinality.

  \begin{theorem}\label{alphagep} Let $\alpha$ be an infinite  cardinal and 
   let $\mathbf{A}=(A,\mathrm{F})$ be an algebra such that $|A| \leq \alpha$.  
Let $i_A:A\to \alpha$ be the respective injective mapping. Let
        $\mathcal{K}$ be a class of algebras of the same type. If
        $\mathbf{A}$ satisfies the general finite $\alpha$-embedding 
        property for $\mathcal{K}$ then there is 
        an $\alpha$-regular ultrafilter over the set 
        $I=\{X; X\subseteq \alpha$\ \mbox{and}\ $X$\
        \mbox{is finite}$\}$ which does 
        not depend on $A$ and algebras $\mathbf{A}_X\in  \mathcal{K}$, 
		$X\subseteq \alpha$ finite 
        such that $A$ can be embedded into $(\prod_{X\in I}\mathbf{A}_X)/U$.
    \end{theorem}

    \begin{proof}
        Let $\mathbf{A}$ satisfy the general finite $\alpha$-embedding  
       property for $\mathcal{K}$.  Then for any $X\in I$ there are 
        $\mathbf{A}_X\in\mathcal{K}$ and an embedding 
        $\rho_X:\mathbf{A}|_{i_{A}^{-1}(X)}\hookrightarrow \mathbf{A_X}$. 
        By the axiom of choice we choose a fixed $a_X\in A_X$ for 
        any $X\in I$. Now we define a mapping 
        $\varphi:\mathbf{A}\rightarrow \prod_{X\in I}\mathbf{A_X}$ by

        $$ \varphi(a)(X) =\left\{
            \begin{array}{ll}
                \rho_X(a) & \mbox{if } i_{A}(a)\in X\\
                a_X & \mbox{otherwise.}
            \end{array}
        \right.$$

        Denote further $U(X)=\{Y;Y\in I \mbox{ and } X\subseteq Y\}$ 
		and $V=\{U(X);X\in I\}$.
        Then for any $U(X),U(Y)\in V$ the equality 
        $U(X)\cap U(Y) = U(X\cup Y)\in V$ holds and
        thus $U(X)\cap U(Y)\not= \emptyset$. Consequently there 
        is an ultrafilter $U$ of $\mathcal{P}(I)$ such that $V\subseteq U$.

        Let us check that $U$ is $\alpha$-regular. Let us put 
        $E=\{U(\{x\})\mid x\in \alpha\}$. Evidently, $|E|=\alpha$ and, 
        for any $X\in I$ we have that $X\in U(\{x\})$ iff  
        $x\in X$. Therefore any $X\in I$ belongs to only finitely 
        many elements of $E$ because $X$ is a finite subset of $\alpha$.

        Hence, we can define a mapping

        $$ \rho:\mathbf{A}\rightarrow (\prod_{X\in I}\mathbf{A}_X)/U$$
        such that $\rho(a)=\varphi(a)/U$.\par
        {(i) \textit{$\rho$ is injective.}} Let $x,y\in A$ be such 
        that $x\not= y$. It follows that $i_{A}(x)\not= i_{A}(y)$ and,
        for any $X\in U(\{i_{A}(x),i_{A}(y)\})$, we 
        have $\varphi(x)(X)=\rho_X(x)\not=\rho_X(y)=\varphi(y)(X)$.
        Hence $U(\{i_{A}(x),i_{A}(y)\})\subseteq \zl \varphi(x)\not= %
    	\varphi(y)\zr \in U$ and finally 
        $\rho(x)=\varphi(x)/U \not= \varphi(y)/U =\rho(y)$.\par
        {(ii) \textit{$\rho$ is a homomorphism.}} Take 
        $f\in \mathrm{F}_n$ and $x_1,\ldots,x_n\in
        A$ such that  $f^\mathbf{A}(x_1,\ldots,x_n)$ is defined. 
        Then, for any 
        $X\in U(i_{A}(x_1),\ldots, i_{A}(x_n),%
        i_{A}(f^\mathbf{A}(x_1,\ldots,x_n))),$ we have

        \begin{eqnarray*}
        \varphi (f^\mathbf{A}\left(x_1,\ldots,x_n)\right)(X) &=%
        & \rho_X(f^\mathbf{A}(x_1,\ldots,x_n))\\
        &=&f^{\mathbf{A}_X}(\rho_X (x_1),\ldots,\rho_X (x_n))\\
        &=&f^{\mathbf{A}_X}(\varphi (x_1)(X),\ldots,\varphi(x_n)(X))\\        
        &=&f^{{\mathbf{\prod}}_{Y\in I}(\mathbf{A}_Y)}%
(\varphi(x_1),\ldots,\varphi(x_n))(X).
        \end{eqnarray*}
        Hence,

        $$U(x_1,\ldots,x_n,f^\mathbf{A}(x_1,\ldots,x_n))\subseteq $$
        $$\zl \varphi (f^\mathbf{A}(x_1,\ldots,x_n))=%
        f^{\mathbf{\prod}_{Y\in I} (\mathbf{A}_Y)}%
        (\varphi(x_1),\ldots,\varphi(x_n))\zr  \in U$$

        \noindent holds. Now we compute

        \begin{eqnarray*}
         \rho (f^\mathbf{A}(x_1,\ldots,x_n)) &=%
         &\varphi (f^\mathbf{A}(x_1,\ldots,x_n))/U \\
        &=&  f^{\mathbf{\prod}_{Y\in I} (\mathbf{A}_Y)}%
        (\varphi (x_1),\ldots,\varphi(x_n))/U \\
        &=& f^{(\mathbf{\prod}_{Y\in I} (\mathbf{A}_Y))/U}%
        (\varphi (x_1)/U,\ldots,\varphi (x_n)/U) \\
        &=& f^{(\mathbf{\prod}_{Y\in I} (\mathbf{A}_Y))/U}%
        (\rho (x_1),\ldots,\rho(x_n)).
        \end{eqnarray*}

        \noindent This shows that $\rho$ is an embedding into 
         $(\prod_{X\in I}\mathbf{A}_X)/U$.%
    \end{proof}

\section{Representation of MV-algebras by regular ultrapowers}

In this section we present a uniform  version of Di Nola's Theorem for 
rationals.  This enables us to embed all MV-algebras of a cardinality at 
most $\alpha$ in an algebra of functions from $2^{\alpha}$ into
a single non-standard ultrapower of the MV-algebra $\mathbb Q\cap [0, 1]$.
Our second goal is to embed all MV-algebras of a cardinality at 
most $\alpha$  into a single non-standard ultrapower of the MV-algebra $(\mathbb Q\cap [0, 1])^{\mathbb N}$.

       \begin{theorem}\label{ulaps} Let $\alpha$ be an infinite  
cardinal and 
   let $\mathbf M=(M;\oplus,\neg,0)$ be a linearly ordered MV-algebra 
   such that $|M| \leq \alpha$, $U$ be the $\alpha$-regular ultrafilter 
   on the set $I=\{X; X\subseteq \alpha$\ \mbox{and}\ $X$\
        \mbox{is finite}$\}$ from Theorem \ref{alphagep} 
 which does not depend on $\mathbf M$.  Then
       \begin{enumerate}
       \item $\mathbf M$ can be embedded into an ultraproduct of finite 
       MV-chains via the $\alpha$-regular ultrafilter $U$.
\item $\mathbf M$ can be embedded into the ultrapower 
 $(\prod_{X\in I}\mathbb Q\cap[0,1])/U$.\phantom{\huge L}
 \item $\mathbf M$ can be embedded into the ultrapower 
 $(\prod_{X\in I} [0,1])/U$.\phantom{\huge L}
       \end{enumerate}
       \end{theorem}
       \begin{proof} (1)-(3) It is a direct corollary 
of Theorem \ref{T-1}, (1) and Theorem \ref{alphagep}.
        \end{proof}

     \begin{theorem}\label{ugenaps} Let $\alpha$ be an 
infinite  cardinal and 
   let $\mathbf M=(M;\oplus,\neg,0)$ be an  MV-algebra 
   such that $|M| \leq \alpha$, $U$ be the $\alpha$-regular ultrafilter 
   on the set $I=\{X; X\subseteq \alpha$\ \mbox{and}\ $X$\
        \mbox{is finite}$\}$  from Theorem \ref{alphagep}  
which does not depend on $\mathbf M$.  Then
       \begin{enumerate}
\item $\mathbf M$ can be embedded into an MV-algebra of functions 
from $2^{\alpha}$ to the ultrapower 
 $(\prod_{X\in I}\mathbb Q\cap[0,1])/U$.
 \item $\mathbf M$ can be embedded into an MV-algebra of functions 
from $2^{\alpha}$ to the ultrapower 
 $(\prod_{X\in I} [0,1])/U$.
       \end{enumerate}
       \end{theorem}
       \begin{proof} (1) Let $i_M:M\to \alpha$ be an injective mapping.
       Let $PFilt(\mathbf M)$ be the set of all prime filters 
       of $\mathbf M$ which is evidently non-empty and let 
       $F_0\in PFilt(\mathbf M)$. Then we have by 
       Chang representation Theorem an embedding 
       $$f:{\mathbf M} \hookrightarrow %
       \prod_{F\in PFilt(\mathbf M)} {\mathbf M}/F.$$
       Moreover, we have an injective mapping  
       $e_M:PFilt(\mathbf M)\to 2^{\alpha}$ given by 
       $F\mapsto \{i_M(x) \mid x\in F\}\subseteq \alpha$. 
       For any $F\in PFilt(\mathbf M)$ we have from 
       Theorem \ref{ulaps} an embedding  
       $$g_F:{\mathbf M}/F \hookrightarrow %
       (\prod_{X\in I}\mathbb Q\cap[0,1])/U.$$
       This yields an embedding 
       $$g:\prod_{F\in PFilt(\mathbf M)} {\mathbf M}/F \hookrightarrow 
\left((\prod_{X\in I}\mathbb Q\cap[0,1])/U\right)^{2^{\alpha}}$$ 
given as follows:
$$g((x_F)_{F\in PFilt(\mathbf M)})(B)=\left\{%
\begin{array}{l l} 
       g_F(x_F)&\text{if}\ e_M(F)=B\\
       g_{F_0}(x_{F_0})&\text{otherwise.}\\
       \end{array}
\right.
$$
The composition of $g\circ f$ gives us the required embedding. \\
(2) It follows by the same considerations as in (1).
        \end{proof}

Going the other way around, we have

     \begin{theorem}\label{sugenaps} Let $\alpha$ be an 
infinite  cardinal and 
   let $\mathbf M=(M;\oplus,\neg,0)$ be a  MV-algebra 
   such that $|M| \leq \alpha$, $U$ be the $\alpha$-regular ultrafilter 
   on the set $I=\{X; X\subseteq \alpha$\ \mbox{and}\ $X$\
        \mbox{is finite}$\}$  from Theorem \ref{alphagep}  
which does not depend on $\mathbf M$.  Then
       \begin{enumerate}
\item $\mathbf M$ can be embedded into the ultrapower 
 $\left(\prod_{X\in I}(\mathbb Q\cap[0,1])^{\mathbb N}\right)/U$.
 \item $\mathbf M$ can be embedded into the ultrapower 
 $\left(\prod_{X\in I}[0,1]^{\mathbb N}\right)/U$.\phantom{\huge L}
       \end{enumerate}
       \end{theorem}
       \begin{proof} (1) Let $i_M:M\to \alpha$ be an injective mapping 
       and let $X\subseteq M$ be a finite subset. Using the same notation 
      and reasonings  as 
       in the proof of Theorem \ref{T-1}, (2) we have an embedding 
       $$f:\mathbf X\hookrightarrow (\mathbf L_k)^l.$$
       Moreover, we have also an embedding 
       $$g:(\mathbf L_k)^l\hookrightarrow (\mathbb Q\cap[0,1])^{\mathbb N}$$ 
       given by:
       $$g((x_k)_{k=1}^{l})(n)=\left\{%
       \begin{array}{l l} 
       x_k&\text{if}\ k=n\\
       x_1&\text{otherwise.}\\
       \end{array}
\right.
$$
The composition $\rho_X=g\circ f$ yields an embedding 
$$\rho_X:\mathbf X \hookrightarrow (\mathbb Q\cap[0,1])^{\mathbb N}.$$ 
The remaining part now follows from Theorem \ref{alphagep}.\\
(2) By the same considerations as in (1).
       \end{proof}

\begin{remark} {\rm Note first that, for a given infinite cardinal $\alpha$, 
$\mathrm{card}%
\left((\prod_{X\in I}\mathbb Q\cap[0,1])/U)^{2^{\alpha}}\right)=%
\aleph_0^{2^{\alpha}}$ and $\mathrm{card}%
\left(\left(\prod_{X\in I}(\mathbb Q\cap[0,1])^{\mathbb N}\right)/U\right)=
2^{\alpha}$.\\[-0.4cm] 

It follows that, for a given infinite cardinal $\alpha$, there is a single 
MV-algebra of the cardinality $2^{\alpha}$ where every MV-algebra of cardinality at most $\alpha$ embeds\\[-0.4cm]

Second, by the same arguments as in \cite[Section 4]{DiNola2}, 
 for every infinite cardinal $\alpha$ there is an iterated 
ultrapower (see  \cite[Section 6.5]{Keisler}) $\prod_{\alpha}$
of $(\mathbb Q \cap[0, 1])^{\mathbb N}$, definable in $\alpha$, 
where every MV-algebra of cardinality at most $\alpha$ embeds.}
\end{remark}

\end{document}